\documentclass[12pt,reqno]{amsart}
\usepackage{amssymb}
\usepackage{mathtools}
\usepackage[english]{babel}
\usepackage{epsfig}
\usepackage{mathptmx}
\usepackage{xcolor}
\usepackage{amsmath,amsfonts,amssymb}
\usepackage{mathtools}
\usepackage{mathrsfs}
\usepackage[all]{xy}
\usepackage{graphicx}
\usepackage{latexsym}
\usepackage{verbatim}
\numberwithin{equation}{section}

\def\1#1{\overline{#1}}
\def\2#1{\widetilde{#1}}
\def\3#1{\widehat{#1}}
\def\4#1{\mathbb{#1}}
\def\5#1{\frak{#1}}
\def\6#1{{\mathcal{#1}}}

\newtheorem{theorem}{Theorem}[section]

\newtheorem{proposition}[theorem]{Proposition}

\theoremstyle{definition}

\theoremstyle{remark}
\newtheorem{remark}[theorem]{Remark}

\numberwithin{equation}{section}

\title{Burns-Krantz rigidity in non-smooth domains}
\author{W\l odzimierz Zwonek}
\address{Institute of Mathematics, Faculty of Mathematics and Computer Science, Jagiellonian University, \L ojasiewicza 6, 30--348 Krak\'ow, Poland}
\email{wlodzimierz.zwonek@uj.edu.pl}
\thanks{Research was partially supported by the Sheng grant no. 2023/48/Q/ST1/00048 of the National Science Center, Poland}
\begin{document}
\begin{abstract} Motivated by recent papers \cite{For-Rong 2021} and \cite{Ng-Rong 2024} we prove a boundary Schwarz lemma (Burns-Krantz rigidity type theorem) for non-smooth boundary points of the polydisc and symmetrized bidisc. Basic tool in the proofs is the phenomenon of invariance of complex geodesics (and their left inverses) being somehow regular at the boundary point under the mapping satisfying the property as in the Burns-Krantz rigidity theorem that lets the problem reduce to one dimensional problem. Additionally, we make a discussion on bounded symmetric domains and suggest a way to prove the Burns-Krantz rigidity type theorem in these domains that however cannot be applied for all bounded symmetric domains.
\end{abstract}

\date{September 2024}

\maketitle

\section{Burns-Krantz rigidity property -- an introduction} Consider the following general problem. Let $D$ be a domain in $\mathbb C^n$ and $p\in\partial D$. Consider a holomorphic mapping $F:D\to D$ that satisfies the property $F(z)=z+o(||z-p||^3)$, $z\in D$ -- in such a situation we say that $F$ satisfies {\it the Burns-Krantz condition at the boundary point $p$}. We want to get {\it Burns-Krantz rigidity type theorem for the pair $(D,p)$} - in other words under which assumptions on $D$ and $p$ one may conclude that $F$ satisfying the Burns-Krantz condition is the identity. Though there are results on the similar property with the exponent other than $3$ we restrict our interest only to this special situation. Let us call the property for $(D,p)$ as {\it the Burns-Krantz rigidity property} (by default with the exponent $3$). In our considerations we focus on Lempert domains (that is taut and such that the Lempert theorem holds). 

We start with some general observation that we then apply in concrete situations. It is also worth mentioning that the interest of the author in the non-smooth case arose after the author has learnt about the results in \cite{For-Rong 2021} and \cite{Ng-Rong 2024} where the Burns-Krantz rigidity type theorem was proven among others for smooth boundary points of the polydisc and bounded symmetric domains. 
%To the astonishment of the author he could not find the %reference for this property for non-smooth boundary points; %even in the case of the bidisc.

The original paper of Burns-Krantz appeared in 1994 (see \cite{Bur-Kra 1994}) and the Burns-Krantz rigidity property was proven for $D$ being a strongly pseudoconvex domain and arbitrary boundary point $p$. Note that at the time of appearence the result was  new even in dimension one and the disc. Later a number of more general results appeared. In our presentation we may find some ressemblence to ideas that can be found in \cite{Hua 1995} and many subsequent papers where the theory of Lempert was used (see e. g. \cite{Lem 1981}, \cite{Lem 1984}). However, unlike the many existing papers we omit the assumption of smoothness of the domain. Instead we concentrate on the regularity assumption of the complex geodesics and left inverses that are guaranteed by the Lempert theorem for sufficiently regular domains (like strongly (linearly) convex ones) - though these properties often hold without the assumption of smoothness of domains. This lets us work in the non-smooth convex setting (more generally Lempert domains) and conclude the Burns-Krantz rigidity property for important domains, like the polydisc (and some bounded symmetric domains) and the symmetrized bidisc. 
%This is astonishing that up to now that case was not discussed in the literature. 
Let us underline that though the method developed by us may also be applied in smooth case we restrict ourselves mostly to the non-smooth case as this seems to give new results. 

Rough idea of our method is to associate with the given complex geodesic the family of its left inverses and then apply the one-dimensional Burns-Krantz rigidity theorem for the disc to conclude that the complex geodesic under the mappings satisfying the Burns-Krantz condition preserves the complex geodesicity and the family of left inverses for complex geodesics.

As a good source for many results in the direction and an updated survey of the known results (in smooth convex case) and potential open problems we recall the paper of Zimmer (\cite{Zim 2022}). As to basics on the theory of holomorphically invariant functions and Lempert theory we refer the readers to \cite{Jar-Pfl 2013}. 

\subsection{Invariance of complex geodesics under maps satisfying the Burns-Krantz condition}
Recall that a domain $D\subset\mathbb C^n$ is called a {\it Lempert domain} if $D$ is taut and the Lempert theorem holds on $D$ (see \cite{Lem 1981}, \cite{Lem 1984}). This may also be formulated so that through any distinct points of $D$ there is a complex geodesic passing through them. The holomorphic mapping $f:\mathbb D\to D$ is called {\it a complex geodesic} if there is a holomorphic function $G:D\to\mathbb D$ such that $G\circ f$ is the identity (in some situations we may only asum $G\circ f$ to be an automorphism of $\mathbb D$). The function $G$ is called of {\it a left inverse for $f$}. By the Lempert theorem (see \cite{Lem 1981}, \cite{Lem 1984}) strongly linearly convex  or bounded convex domains are Lempert domains. We also know that the symmetrized bidisc and the tetrablock are Lempert domains (see \cite{Agl-You 2004}, \cite{Cos 2004}, \cite{Edi-Kos-Zwo 2013}).

Consider a complex geodesic $f:\mathbb D\to D$ that is Lipschitz continuous at $1$, i. e. $f$ extends continuously to $1$ with $f(1)=p$ and  for an open neighborhood $U$ of $1$ the mapping $f$ is Lipschitz continuous on a set $U\cap (\mathbb D\cup\{1\})$. Consider also a left inverse $G:D\to\mathbb D$ for $f$ that is Lipschitz continuous at $p$. Then $G(f(\lambda))=\lambda$, $\lambda\in\mathbb D$ (and then $G(p)=1$).

As already announced we formulate a result that instead of the regularity assumption of the domain lets us work with regular (Lipschitz) complex geodesics and left inverses.

One should mention that the idea of the proof of Proposition~\ref{proposition:invariance-left-inverse} may be essentially found in the proof of Theorem 2.5 in \cite{Hua 1995} (compare also Corollary 2 from \cite{Hua 1994}).

\begin{proposition}\label{proposition:invariance-left-inverse}
    
    Let $D$ be a domain in $\mathbb C^n$, $p\in\partial D$ and let $F:D\to D$ be holomorphic that satisfies the property $F(z)=z+o(||z-p||^3)$. Assume that $f:\mathbb D\to D$ is a complex geodesic that is Lipschitz continuous at $1$ and $f(1)=p$ and let $G:D\to\mathbb D$ be a left inverse to $f$ that is Lipschitz continuous at $p$. Then $G$ is a left inverse to $F\circ f$; in particular, $F\circ f$ is a complex geodesic.
\end{proposition}

\begin{proof}
    By the assumptions for $\lambda\in\mathbb D$ close to one we have
\begin{multline}
    G(F(f(\lambda)))-\lambda=G(F(f(\lambda)))-G(f(\lambda))=O(F(f(\lambda))-f(\lambda))=\\
    o(||f(\lambda)-p||^3)=o(||f(\lambda)-f(1)||^3)=o(|\lambda-1|^3).
\end{multline}
It is sufficient to apply the Burns-Krantz rigidity theorem for the disc to get the conclusion.
\end{proof}

\begin{remark}
    Note that the above proposition leads in a natural way to a {\it slice rigidity property} as defined and discussed in \cite{Bra-Kos-Zwo 2021}. Roughly speaking under which assumptions on the domain a holomorphic mapping preserving Kobayashi isometrically complex geodesics from a complete foliation of the domain with the common boundary point must be a biholomorphism.  
    
    It also should be remarked that the case of strongly linearly convex domains lets the existence of smooth complex geodesics joining boundary points and providing the existence of left inverses as proven by Lempert - as the 'regular' left inverse means also that it assumes the values at $\partial \mathbb D$ only along the boundary of the complex geodesic $f(\partial\mathbb D)$ (see e. g. \cite{Cha-Hu-Lee 1988}, \cite{Lem 1981}). That kind of property gives in such a situation that $F$ is the identity when restricted to the image of complex geodesics with the given boundary point. We shall however be interested in applying that method in the cases where no such regularity behaviour is guaranteed by Lempert proofs.
\end{remark}

We apply the result presented in Proposition~\ref{proposition:invariance-left-inverse} in two cases. First we use it in the basic case of the polydisc (that admits quite few left inverses) and then in the case of the symmetrized bidisc (that also has a relatively small (in some sense) and well understood family of left inverses). In the case of the symmetrized bidisc a key role is played by the understanding of (non)-uniqueness of left inverses of complex geodesics (that problem was studied in detail in \cite{Kos-Zwo 2016}).
We also remark that a possible application may go far beyond the examples we study. In particular, a research may be continued to employ results on (uniqueness and regularity) of left inverses not only in the symmetrized bidisc but in a wider class of domains (like the tetrablock).
This is also tempting to apply in the future the developed method while dealing with the Burns-Krantz rigidity property for a wider class of domains.

\section{Polydisc}
We start with the proof of the Burns-Krantz rigidity property for the polydisc.

\begin{remark}
    In the proofs below we shall make repeatedly use of the following property that may be proven directly or is implicitly mentioned in \cite{Kos-Zwo 2016} (Theorem 4.1) or \cite{Agl-Lyk-You 2019b}. Consider, the complex geodesic
    \begin{equation}
        \mathbb D\owns\lambda\to(\lambda,a_2(\lambda),\ldots,a_n(\lambda))\in\mathbb D^n,
    \end{equation}
    where $a_2,\ldots,a_n$ are holomorphic self-mappings of $\mathbb D$ that are not automorphisms. Then the complex geodesic
    has exactly one, uniquely determined left inverse which is the projection on the first coordinate. 
    
    %Moreover, for our purposes we need the functions $a_j$ to be chosen %so that they are holomorphic on a neighborhood of the closure of %$\mathbb D$ (which is possible). This would make the complex %geodesics be Lipschitz and the above defined complex geodesics would %satisfy the assumptions of Proposition~\ref{proposition:invariance-%left-inverse}.  
\end{remark}

\begin{theorem}\label{theorem:burns-krantz-polydisc}
    The Burns-Krantz rigidity theorem holds for the polydisc $\mathbb D^n$ and any boundary point $p\in\partial \mathbb D^n$.
\end{theorem}
\begin{remark}
    Recall that the results of \cite{For-Rong 2021} imply the above theorem for smooth boundary points $p$.  Actually, in our proof we rely on Proposition~\ref{proposition:invariance-left-inverse} in non-smooth points that combined with results from \cite{For-Rong 2021} gives a complete proof.
\end{remark}

\begin{proof}
Let $F:\mathbb D^n\to\mathbb D^n$ be a holomorphic mapping satisfying the assumption $F(z)=z+o(||z-p||^3)$ where $p$ is a boundary point of $\mathbb D^n$.
Without loss of generality we may assume that $|p_j|=1$, $j=1,\ldots,k$, $|p_j|<1$, $j=k+1,\ldots,n$ where $1\leq k\leq n$. 

First we prove that $F_j(z)=z_j$, $j=1,\ldots,k$.

Certainly, it is sufficient to prove that $F_1(z)=z_1$. 

Consider the (many) holomorphic mappings $a_j:\mathbb D\to\mathbb D$, $j=2,\ldots,n$ extending holomorphically through $1$ such that $a_j(1)=p_j$ and $a_j$ is not an automorphism of $\mathbb D$, $j=2,\ldots,n$. It is elementary to see that for the fixed element $\lambda\in\mathbb D$ the family of all values of $\{(a_2(\lambda),\ldots,a_n(\lambda))\}$ over all such $a_j$'s is $\mathbb D^{n-1}$. 

The mapping $f:\mathbb D\owns\lambda\to (\lambda,a_2(\lambda),\ldots,a_n(\lambda))\in\mathbb D^n$ is a complex geodesic with the (uniquely determined) left inverse $z_1$ and by Proposition~\ref{proposition:invariance-left-inverse} the composition $F\circ f$ must be a complex geodesic with the same left inverse. In other words $F_1(\lambda,a_2(\lambda),\ldots,a_n(\lambda))=\lambda$ for all $\lambda\in\mathbb D$ and all possible functions $a_j$ as determined above. This gives that $F_1(z)=z_1$ as claimed.

In the case $k=n$ the above finishes the proof. Assume then that $1\leq k<n$.

Consider the mapping
\begin{equation}
    G:\mathbb D^{n-k+1}\owns (z_1,z_{k+1},\ldots,z_n)\mapsto F_{(1,k+1,\ldots,n)}(z_1,\ldots,z_1,z_{k+1},\ldots,z_n)\in\mathbb D^{n-k+1}.
    \end{equation}
Note that $G$ satisfies the assumption of the Burns-Krantz theorem at the smooth boundary point $(p_1,p_{k+1},\ldots,p_n)$ from $\partial\mathbb D^{n-k+1}$. By a result from \cite{For-Rong 2021} we get that $G$ is the identity. To finish the proof we need to conclude that $F$ is the identity, too. 

In other words it is sufficient to show that the following Claim holds.

{\bf Claim} Let $H:\mathbb D^n\to\mathbb D$ be a holomorphic function such that $H(z_1,\ldots,z_1,z_n)=z_n$ for all $(z_1,z_n)\in\mathbb D^{2}$. Then $H(z)=z_n$, $z\in\mathbb D^n$.

Take arbitrary point $z=(z_1,z_2,\ldots,z_n)\in\mathbb D^n$. Fix $\alpha,\beta\in\mathbb D$. It follows from one-dimensional interpolation problem  that one may find distinct numbers $v_n,w_n\in\mathbb D$ sufficiently close to distinct boundary points of $\partial \mathbb D$ (or the ones with big enough Poincar\'e distance) and holomorphic mappings $a_j:\mathbb D\to\mathbb D$, $j=1,\ldots,n-1$ with $a_j(v_n)=\alpha$, $a_j(w_n)=\beta$ and $a_j(z_n)=z_j$, $j=1,\ldots,n-1$. Consider the holomorphic function 
$$\mathbb D\owns g:\lambda\to H(a_1(\lambda),\ldots,a_{n-1}(\lambda),\lambda)\in\mathbb D.
$$
As $g(v_n)=H(\alpha,\ldots,\alpha,v_n)=v_n$ and $g(w_n)=H(\beta,\ldots,\beta,w_n)=w_n$ we get by the Schwarz Lemma that $g(\lambda)=\lambda$, $\lambda\in\mathbb D$. In particular, $H(z)=g(z_n)=z_n$.

 %{\bf Claim} Let $H:\mathbb D^n\to\mathbb D$ be a holomorphic function %such that $H(z_1,\ldots,z_1,z_{k+1},\ldots,z_n)=z_n$ for all %(z_1,z_{k+1},\ldots,z_n)\in\mathbb D^{n-k+1}$. Then $H(z)=z_n$, %$z\in\mathbb D^n$.

%To prove the Claim for fixed distinct points %$\lambda_1,\lambda_2\in\mathbb D$ consider families of holomorphic %functions $a_j:\mathbb D\to\mathbb D$, $j=1,\ldots,k$ with %$a_1(\lambda_j)=\ldots=a_k(\lambda_j)$, $j=1,2$ and holomorphic functions %$a_{k+1},\ldots,a_{n-1}:\mathbb D\to\mathbb D$. Then for any fixed %$\lambda\in\mathbb D$ the set of elements $\{(a_1(\lambda),\ldots,a_{n-1}
%(\lambda))\}$ over all such distinct $\lambda_1,\lambda_2$ and all %$a_m$'s is $\mathbb D^{n-1}$. As $H(a_1(\lambda_j),\ldots,a_{n-1}%(\lambda_j),\lambda_j))=\lambda_j$ by the Schwarz Lemma we get that %$H(a_1(\lambda),\ldots,a_{n-1}(\lambda),\lambda)=\lambda$, %$\lambda\in\mathbb D$. As the equality holds for systems of %$\lambda_1,\lambda_2$ and functions $a_m$'s as above we get %$H(z,\lambda)=\lambda$, $(z,\lambda)\in\mathbb D^{n-1}\times\mathbb D$ %and finish the proof of the Claim.
\end{proof}

\section{The symmetrized bidisc} We continue applying our method to get the property in one more special domain. {\it The symmetrized bidisc} defined as 
\begin{equation}
\mathbb G_2:=\{(\lambda_1+\lambda_2,\lambda_1\lambda_2):\lambda_1,\lambda_2\in\mathbb D\}
\end{equation}
has proven to be an essential example in the Lempert theory. It is a Lempert domain that is not biholomorphic to a convex domain (see \cite{Agl-You 2001}, \cite{Agl-You 2004} and \cite{Cos 2004}). Moreover, it has some interesting properties that will play a role in the application of the method introduced earlier while proving the Burns-Krantz property; namely, all complex geodesics of $\mathbb G_2$ extend holomorphically through the boundary (compare the formulas for complex geodesics in $\mathbb G_2$ -- \cite{Agl-You 2006} and \cite{Pfl-Zwo 2005}). Moreover, a complete description of the uniqueness property of left inverses for complex geodesics in $\mathbb G_2$ is known. The left inverses are also regular (but only to some extent) -- that would be made clearer later which would make possible to make use of Proposition~\ref{proposition:invariance-left-inverse}. A good example of {\it the universal Carath\'eodory set} in the symmetrized bidisc is a class of functions of the form:
\begin{equation}
    \Psi_{\omega}(s,p):=\frac{2p-\omega s}{2-\overline{\omega}s},\; (s,p)\in\mathbb G_2, \; \omega\in\overline{\mathbb D}.
\end{equation}
Recall that the universal Carath\'eodory set is
 a family of functions that could replace the class of all bounded by one holomorphic functions in the definition of the Carath\'eodory distance or in the case of Lempert domains it may be understood as the family of functions out of which one may find for any complex geodesic in the domain a left inverse.
 
The minimal Carath\'eodory universal set for the symmetrized bidisc exists and consists of functions $\Psi_{\omega}$, $|\omega|=1$ (see \cite{Agl-Lyk-You 2019b}).

Note also that the functions $\Psi_{\omega}$ extend holomorphically through all the boundary points of $\mathbb G_2$ with the exception for $|\omega|=1$. Namely, in this case the function $\Psi_{\omega}$ does not extend even continuously through one boundary point $(2\omega,\omega^2)$.

Let us also draw attention to another special property that differs the symmetrized bidisc from the bidisc. The boundaries of all complex geodesics of the symmetrized bidisc are lying in the Shilov boundary of $\mathbb G^2$ -- recall that the Shilov boundary of $\mathbb G_2$ is $\{(\lambda_1+\lambda_2,\lambda_1\lambda_2):\lambda_1,\lambda_2\in\partial \mathbb D\}$. In other words no other element of the topological boundary may be joint by a geodesic with the interior point. This is the reason why our method of proving the Burns-Krantz rigidity property may work only for points from the Shilov boundary of the symmetrized bidisc.

We prove the following
\begin{theorem}\label{theorem:burns-krantz-property-symmetrized-bidisc} The Burns-Krantz rigidity property holds for $(\mathbb G_2,w)$ for any $w$ from the Shilov boundary of $\mathbb G_2$.
\end{theorem}

Before we go into the proof let us make remarks how we apply our method while showing the Burns-Krantz rigidity property at Shilov boundary points of $\mathbb G_2$. Recall that in the paper \cite{Kos-Zwo 2016} a complete description of uniqueness property of left inverses in the symmetrized bidisc is given; additionally, in the case of the non-uniqueness a description of left inverses of the form $\Psi_{\omega}$ is found. Recall also that the left inverses may be meant to be understood as functions up to a composition with an automorphism of the unit disc, i. e. the composition of the left inverse with the complex geodesic equals the identity up to an automorphism of the unit disc. 

%Note that all the complex geodesic in the symmetrized bidisc %touch the boundary only at the points of the Shilov boundary %of $\mathbb G_2$ -- this is completely different from the %case of the bidisc. This is also the reason why our method of %the proof may work only for the points from the Shilov %boundary of $\mathbb G_2$. 
While proving the theorem we may consider only points from the Shilov boundary of the form $(1+\omega,\omega)$ (with $|\omega|=1$) -- the consideration of such points does not restrict the generality of our proof. 

The rough idea of the proof is the following. The complex geodesics touching the boundary at the point $(1+\omega,\omega)$, $|\omega|=1$ are mapped by the mapping $F$ satisfying the Burns-Krantz condition at $(1+\omega,\omega)$ to complex geodesics having the left inverses as the original one (with possibly new ones). This is a general idea that encounters an obstacle as the existing left inverse may not satisfy the assumption of Proposition~\ref{proposition:invariance-left-inverse}. That obstacle turns up in some cases and is overcome suitably.

%Therefore, we choose another direct way of this %reasoning that will overcome this obstacle. We %start with the boundary point %$(1+\omega,\omega)$ with $|\omega|=1$, %$\omega\neq 1$. 

\begin{remark}\label{remark:description-left-inverses} A general idea of the proof relies on comparing sets of left inverses of the form $\Psi_{\omega}$ that are related to the given complex geodesic in $\mathbb G_2$.  Below we summarize the results of \cite{Kos-Zwo 2016} (Theorem 5.4 and remarks before it) which comprise the results focusing mainly to the complex geodesics we shall be interested in.

A very special complex geodesic in the symmetrized bidisc is {\it the royal geodesic} (and its image is called {\it the royal variety}) 
\begin{equation}
\mathbb D\owns\lambda\to(2\lambda,\lambda^2)\in\mathbb G_2.
\end{equation}
Then all the functions $\Psi_{\omega}$, $|\omega|= 1$, are its left inverses. Moreover, the royal geodesic is the only one having as left inverses $\Psi_{\omega}$'s, $|\omega|=1$ and such that no $\Psi_{\omega}$ (with $|\omega|<1$) is its left inverse.

For all $\beta\in\mathbb D$ the mappings 
\begin{equation}
    \mathbb D\owns\to(\beta+\overline{\beta}\lambda,\lambda)\in\mathbb G_2
    \end{equation}
are so called {\it flat geodesics} (see \cite{Agl-Lyk-You 2019}) -- the flat geodesics are generated by geodesics of the form $(0,\lambda)=-\pi(-\sqrt{\lambda},\sqrt{\lambda})$ as considered in \cite{Pfl-Zwo 2005} and then composed with automorphisms of the symmetrized bidisc. All functions $\Psi_{\omega}$, $|\omega|\leq 1$ are left inverses of all flat geodesics.

We also need to know that the complex geodesics intersecting the royal variety at exactly one point and not being the flat geodesics have exactly one left inverse (being a function $\Psi_{\omega}$). 

In our proof we shall also be interested in complex geodesics passing through $(0,0)$ and touching the boundary at $(2,1)$. By the description of complex geodesics they must be of the following form (up to a composition with an automorphism of the unit disc) - use for instance Theorem 1 and Remark afterwards in \cite{Pfl-Zwo 2005}:
\begin{equation}
      f_t(\lambda):=\left(2\lambda\frac{1-t}{1-t\lambda},\lambda\frac{\lambda-t}{1-t\lambda}\right),\; t\in(0,1),\;\lambda\in\mathbb D.
\end{equation}
One may easily verify that $-\Psi_1(f_t(\lambda))=\lambda$, $t\in(0,1)$, $\lambda\in\mathbb D$. As already mentioned $\Psi_1$ is the only left inverse for all $f_t$, $t\in(0,1)$.

It should also be remarked that there is one more class of complex geodesics: the ones omitting the royal variety. Though they will not be so essential in the proof we shall use the fact that no more than two different $\Psi_{\omega}$'s (with both $|\omega|=1$) are left inverses of such geodesics.
\end{remark}

\begin{proof}[Proof of Theorem~\ref{theorem:burns-krantz-property-symmetrized-bidisc}]
As already mentioned we lose no generality assuming that the boundary point is $(1+\omega,\omega)$, $|\omega|=1$. Let $F:\mathbb G_2\to \mathbb G_2$ be a holomorphic mapping satisfying the Burns-Krantz condition at $(1+\omega,\omega)\in\partial\mathbb G_2$.

At first we consider $\omega\neq 1$. For $\beta\in\mathbb D$
such that $\beta+\overline{\beta}\omega=1+\omega$ (there are 'many' such $\beta$'s) consider the function:
\begin{equation}
    g_{\beta}:\mathbb D\owns\lambda\to F_2(\beta+\overline{\beta}\lambda,\lambda)\in\mathbb D.
\end{equation}
Actually, we may exactly determine what 'many' $\beta$'s as described above means. They are of the form $1+r\sqrt{-\omega}$ where $\sqrt{-\omega}$ is one of possible roots and $r$ are either positive or negative real numbers with arbitrary small absolute value (so that $|1+r\sqrt{-\omega}|<1$.

Assumption on $F$ implies that 
 \begin{multline}
F(\beta+\overline{\beta}\lambda,\lambda)=(\beta+\overline{\beta}\lambda,\lambda)+o(||(\beta+\overline{\beta}\lambda-(1+\omega),\lambda-\omega)||^3)=\\(\beta+\overline{\beta}\lambda,\lambda)+o(|\lambda-\omega|^3)
 \end{multline}
 In particular,
 $g_{\beta}$ satisfies the Burns-Krantz property at $\omega$ which implies that $g_{\beta}$ is the identity. Thus $F_2(\beta+\overline{\beta}\lambda,\lambda)=\lambda$ for all $\lambda\in\mathbb D$ and $\beta$ from a 'big' set (depending on $\omega$). The identity principle gives $F_2(s,p)=p$, $(s,p)\in\mathbb G_2$. To finish the proof we need to verify that $F_1(s,p)=s$, $(s,p)\in\mathbb G_2$. We proceed as follows. By comparing left inverses to complex geodesics given in Remark~\ref{remark:description-left-inverses} and applying Proposition~\ref{proposition:invariance-left-inverse} we get that for any $\beta\in\mathbb D$ the equality $F_1(\beta+\overline{\beta}\lambda,\lambda)=\gamma+\overline{\gamma}\lambda$ holds for any $\lambda\in\mathbb D$ where $\gamma+\overline{\gamma}\omega=1+\omega$ for some $\gamma\in\mathbb D$. From the formula above we get that $\gamma+\overline{\gamma}\lambda=F_1(\beta+\overline{\beta}\lambda,\lambda)=\beta+\overline{\beta}\lambda+o((\lambda-\omega)^3)$ which gives that $\gamma=\beta$ or $F_1(\beta+\overline{\beta}\lambda,\lambda)=\beta+\overline{\beta}\lambda$, $\lambda\in\mathbb D$ and $\beta$ from the 'big' set. Then the identity principle easily finishes the proof.

Consider now the Shilov boundary point $(2,1)$. Recall that the royal geodesic $\mathbb D\owns\lambda\to(2\lambda,\lambda^2)\in\mathbb G_2$ is the geodesic for which all the functions $\Psi_{\omega}$, $|\omega|= 1$ are left inverses. By Proposition~\ref{proposition:invariance-left-inverse} we get consequently that $F(2\lambda,\lambda^2)$ is a geodesic with all functions $\Psi_{\omega}$, $|\omega|=1$, $\omega\neq 1$ being its left inverse. As $F(2\lambda,\lambda^2)\to(2,1)$ as $\lambda\to 1$ and no flat geodesic 'touches' the boundary point $(2,1)$ we get by Remark~\ref{remark:description-left-inverses} that $F$  leaves the royal geodesic invariant so $F(2\lambda,\lambda^2)=(2\lambda,\lambda^2)$, $\lambda\in\mathbb D$. In particular, $F(0,0)=(0,0)$. We shall look now at the transformation of other complex geodesics joining $(0,0)$ with $(2,1)$ under the mapping $F$. We cannot apply directly Proposition~\ref{proposition:invariance-left-inverse} as the only (in the case the complex geodesic passing through the origin is neither the royal nor the flat geodesic) left inverse $\Psi_1$ does not extend continuously through $(2,1)$. However, we apply the idea presented in that proposition directly to see that for such complex geodesics $f$ (with $f(0)=(0,0)$, $f(1)=(2,1)$) the transformation $F\circ f$ will also be a geodesic with the same left inverse $\Psi_1$ and then $F\circ f$ would equal $f$ after taking into account some more information. 

Recall that the complex geodesics under consideration are of the form
\begin{equation}
    f_t(\lambda):=\left(2\lambda\frac{1-t}{1-t\lambda},\lambda\frac{\lambda-t}{1-t\lambda}\right),\; t\in(0,1),\;\lambda\in\mathbb D.
\end{equation}
We also have $-\Psi_1(f_t(\lambda))=\lambda$, $\lambda\in\mathbb D$, $t\in(0,1)$. We look at the function $\mathbb D\owns\lambda\to -\Psi_1(F(f_t(\lambda)))$. By the calculations we get that the last expression near $1$ behaves like 
\begin{equation}
    \lambda+o((\lambda-1)^2).
\end{equation}
Since $-\Psi_1(F(f_t(0)))=0$ we get 
by Lemma 2.1 from \cite{Hua 1995} that $-\Psi_1(F(f_t(\lambda)))=\lambda$, $\lambda\in\mathbb D$. This means that $F\circ f_t$ is a complex geodesic $f_s$ for some $s\in[0,1]$ (for $s=0$ we have a royal geodesic and $s=1$ a flat geodesic).

Calculating directly we also get that $f_t^{\prime}(1)=2/(1-t)(1,1)$.

The assumption on $F$ easily implies that the derivatives at $1$ of $f_t$ and $f_s$ must be equal so $t=s$, which gives $F(f_t(\lambda))=f_t(\lambda)$, $t\in(0,1)$, $\lambda\in\mathbb D$.

The identity principle implies that $F$ is the identity.

\end{proof}

\section{Bounded symmetric domains - discussion}
It is quite natural to proceed further with more general non-smooth domains, a good source of examples would be bounded symmetric domains in its Harish-Chandra realization. We show below a natural attitude to that problem that however cannot be applied in full generality (for all bounded symmetric domains). To do it we make use of the polydisc theorem that reduces the problem to the case of the polydisc. We follow the idea presented in the proof of Lemma 3.1 of \cite{Ng-Rong 2024}.

\begin{remark} Fix a boundary point $p$ of the bounded symmetric domain $D\subset\mathbb C^n$. For any point $z\in D$ consider a Lipschitz continuous complex geodesic $f$ joining $z$ and $p$. The fact that there are such geodesics follows from the symmetry of the domain: the domain is balanced and convex and the involutive automorphisms extend holomorphically through the boundary. By the polydisc theorem we find an $r$-dimensional polydisc $P$  ($r$ is the rank of the domain $D$) embedded in $D$ with the graph of the geodesic $f$ lying in $P$. Taking the automorphism of $D$ we may assume that $P$ is $\mathbb D^r\times\{0\}^{n-r}$, $z=0$ and the orthogonal projection maps $D$ to $P$. Additionally, $f(\lambda)=(\lambda,f_2(\lambda),\ldots,f_r(\lambda),0,\ldots,0)$, $\lambda\in\mathbb D$ with $f_j:\mathbb D\to\mathbb D$ being holomorphic, Lipschitz continuous at $1$, $f_j(0)=0$, $j=2,\ldots,r$. 

Let $\pi$ be the projection of $\mathbb C^n$ onto $\mathbb C^r\times\{0\}^{n-r}$. 

Let us make the following additional assumption

\begin{equation}
    \text{ all the boundary points of $D$ except for the ones in $\partial P$ are mapped to $P$.}\label{ref:assumption} 
\end{equation}

\bigskip

By Theorem~\ref{theorem:burns-krantz-polydisc} applied to $P\owns w\to \pi\circ F(w)\in P$ we know that $\pi(F(w))=w$ or $F(f(\lambda))=(\lambda,f_2(\lambda),\ldots,f_n(\lambda))$, where new holomorphic functions $f_{r+1},\ldots,f_n:\mathbb D\to\mathbb D$ appear. Now the assumption (\ref{ref:assumption}) on $D$ implies that $\lim_{\lambda\to \partial\mathbb D}f_j(\lambda)=0$, $j=r+1,\ldots,n$ which implies $f_{r+1}\equiv\ldots\equiv f_n\equiv 0$ and consequently shows that $F$ when restricted to the graph $f$ is the identity so $F(0)=0$. As the point $z\in D$ that we started with was arbitrary this finishes the proof.
\end{remark}

\begin{remark}
    Note that we cannot hope the above assumption (\ref{ref:assumption}) to hold in all bounded symetric domains. Let $\Omega(n)$ be a Cartan domain of the first type $I_{n,n}$. Its rank is $n$. Then the projection
    \begin{equation}
    \Omega(n)\owns A\to (a_{11},\ldots,a_{nn})\in\mathbb D^n
    \end{equation}
    does not satisfy assumption (\ref{ref:assumption}) for $n\geq 3$ though the condition is satisfied for $n=2$.
\end{remark}

\section{Concluding remarks, possible area of future research and some open problems}

We presented approach to the study of the Burns-Krantz type rigidity theorem in Lempert domains without the smoothness assumption of the domain which may lead to the further research.

Note that though the problem for the polydisc has been solved it is interesting whether we could use the above developed method also to smooth boundary points of the polydisc so that we would not refer to other papers.

A limitation on the case of the symmetrized bidisc only to points from its Shilov boundary (in other words non-smooth ones) depended on a very specific boundary behavior of complex geodesics. This is however natural to pose the question whether the Burns-Krantz rigidity theorem holds also for smooth boundary points of the symmetrized bidisc.

A quite natural challenge would be to see whether the Burns-Krantz rigidity property may be generalized to the symmetrized polydisc. Note that the method we make use of is not applicable in that situation as the higher dimensional symmetrized polydisc is not a Lempert domain.

That would also be intersting to give more examples of new examples of domains with the Burns-Krantz rigidity property; especially possibly utilizing the above mentioned method. More concretely, what about the tetrablock or its higher dimensional generalization that is a domain $\mathbb L_n$ (see \cite{Gho-Zwo 2023})?

\section{Acknowledgments} The author would like to thank professor Feng Rong for his comments on the early version of the paper.

\end{document}